\documentclass[]{amsart}

\usepackage{amssymb,amscd,amsmath, newlfont,color,latexsym}
\usepackage[latin2]{inputenc}
\usepackage[dvips]{graphicx}
\usepackage[cmtip,arrow]{xy}
\usepackage{pb-diagram, pb-xy}

\def\R{\mathbb{R}}
\def\C{\mathbb{C}}
\def\N{\mathbb{N}}

\def\BB{\mathrm{B}}

\def\MM{\mathrm{M}}

\def\E{\mathrm{E}}

\def\im{\mathop{\mathrm{Im}}}

\newtheorem{theorem}{Theorem}[section]
\newtheorem{lemma}[theorem]{Lemma}
\newtheorem{corollary}[theorem]{Corollary}
\newtheorem{proposition}[theorem]{Proposition}
\theoremstyle{remark}
\newtheorem{remark}[theorem]{Remark}
\theoremstyle{definition}
\newtheorem{definition}[theorem]{Definition}
\newtheorem{example}[theorem]{Example}
\theoremstyle{question}
\newtheorem{question}[theorem]{Question}
\numberwithin{equation}{section}

\begin{document}

\title[]{Derivations which are inner as completely bounded maps}
\author{Ilja Gogi\'c}
\address{Department of Mathematics\\ University of Zagreb\\
Bijeni\v{c}ka cesta 30\\ Zagreb 10000\\ Croatia}
\email{ilja@math.hr}

\thanks{The author wishes to thank Professors Bojan Magajna and Douglas W. B.
Somerset for their comments and suggestions.}


\begin{abstract} We consider derivations from the image of the canonical
contraction $\theta_A$ from the Haagerup tensor product of a $C^*$-algebra $A$
with itself to the space of completely bounded maps on $A$. We show that such
derivations are necessarily inner if $A$ is prime  or if $A$ is
quasicentral with Hausdorff primitive spectrum. We also provide an
example of a $C^*$-algebra which has outer elementary derivations.
\end{abstract}

\maketitle

\section{Introduction}

Let $A$ be a $C^*$-algebra and let ICB($A$) be the space of all completely
bounded maps $T: A \to A$ such that $T(J) \subseteq J$, for every closed
two-sided ideal $J$ of $A$. If $A \otimes_h A$ denotes the Haagerup tensor
product of $A$ with itself, there is a canonical contraction $\theta_A : A
\otimes_h A \to \mathrm{ICB}(A)$ given on elementary tensors $a \otimes b \in A
\otimes A$ by
$$\theta_A(a \otimes b)(x):=axb, \quad \mathrm{for} \ \mathrm{all} \ x \in A.$$
Mathieu showed that $\theta_A$ is isometric if and only if $A$ is a prime
$C^*$-algebra (see \cite[5.4.11]{Ara}).
If $A$ is not prime then $\theta_A$ is not even injective, and then it is natural
to consider the central Haagerup tensor product
$A \otimes_{Z,h}A$, and the induced contraction $\theta^Z_{A} : A \otimes_{Z,h} A
\to \mathrm{ICB}(A)$ (see \cite{Som2}, \cite{AST} and \cite{AST1} for the
further details and results in this subject).

Since every derivation on a $C^*$-algebra $A$ is an operator in
$\mathrm{ICB}(A)$, it is natural to study how large can the set
$\mathrm{Der}(A) \cap \im \theta_A$ be (where $\im \theta_A$ denotes the image of
$\theta_A$ and $\mathrm{Der}(A)$ the space of all derivations on $A$).
To ensure that at least all the inner derivations on $A$ are in $\im
\theta_A$ (A is not necessarily unital), we shall require that $A$ is quasicentral (see section 3). In this paper we shall be mainly interested in the question when is
the set $\mathrm{Der}(A) \cap \im \theta_A$ as small as possible, and hence (in
the quasicentral case) equal to the set $\mathrm{Inn}(A)$ of all inner derivations
on $A$. This is certainly true for all von Neumann algebras (since by Kadison-Sakai theorem \cite[4.1.6]{Sak}, every derivation
on a von Neumann algebra is inner). As we shall see, this
property is also satisfied for the class of
all unital prime $C^*$-algebras and for the class of all quasicentral $C^*$-algebras with Hausdorff
primitive spectrum. We also conjecture that this property holds for the
larger class of all quasicentral $C^*$-algebras in which every Glimm ideal is primal, but we
were not able to prove this.

The paper is organized as follows. In section 3 we provide some basic facts
about quasicentral $C^*$-algebras.

In Section 4, we concentrate on prime $C^*$-algebras. We show that every
derivation $\delta \in \im \theta_A$ on a unital prime $C^*$-algebra $A$ is
necessarily inner in $A$. If a prime $C^*$-algebra $A$ is non-unital (and hence
non-quasicentral) we show that the only derivation $\delta \in \im \theta_A$ is in
fact the zero-derivation.

In Section 5, we concentrate on $C^*$-algebras with Hausdorff primitive
spectrum. We show that every derivation $\delta \in \im \theta_A$ is smooth (see
definition \ref{smooth}), and hence inner in its multiplier algebra $M(A)$. If
$A$ is also quasicentral, we prove that every derivation $\delta \in \im \theta_A$
is in fact inner in $A$. We also show that a quasicentral $C^*$-algebra $A$ has a
Hausdorff primitive spectrum if and only if every inner derivation on $A$ is
smooth.

In Section 6, we give an example of a unital separable $2$-subhomogeneous
$C^*$-algebra $A$ for which the space of elementary operators $\E(A)$
is a closed subspace of $\mathrm{ICB}(A)$ (and hence $\im \theta_A=\E(A)$), but
for which the space of inner derivations is not closed in $\mathrm{Der}(A)$. It
follows that such $C^*$-algebra must have outer elementary derivations.

\section{Notation and Preliminaries}
 Through this paper $A$ will denote a $C^*$-algebra, $A_+$ the positive part and $A_h$ the self-adjoint part of $A$. By $Z(A)$ we denote the center of $A$. By an
ideal of $A$ we shall always mean a closed two-sided ideal. The set of all
ideals of $A$ is denoted by $\mathrm{Id}(A)$. By $\hat{A}$ we shall denote the
spectrum of $A$ (i.e. the set of all equivalence classes of irreducible
representations of $A$) and by
$\mathrm{Prim}(A)$ the primitive spectrum of $A$ (i.e. the set of all primitive
ideals of $A$), equipped with the Jacobson topology. By $M(A)$ we denote the multiplier algebra of $A$ and by $\tilde{A}$ we denote the (minimal) unitization of $A$ (whether $A$ is unital or not).
\medskip

We now recall some properties about the complete regularization of Prim($A$)
(see \cite{ArcSom1} for further details). For $P, Q \in \mathrm{Prim}(A)$ let
\begin{equation}\label{approx}
P \approx Q \ \ \mathrm{if }\ \ f(P)=f(Q), \ \mathrm{for} \ \mathrm{all} \ f \in
C_b(\mathrm{Prim}(A)).
\end{equation}

Then $\approx$ is an equivalence relation on Prim($A$) and the equivalence
classes are closed subsets of Prim$(A)$. Hence the map
$$\mathrm{Prim}(A)/\approx \to \mathrm{Id}(A), \quad [P]\mapsto \bigcap[P]$$ is
an injection; here $[P]$ denotes the equivalence class of $P \in
\mathrm{Prim}(A)$. The image of this map is denoted by Glimm$(A)$, and its
elements are called \textit{Glimm ideals} of $A$. We equip Glimm($A$) with the
weakest topology that makes all functions $f \in C_b(\mathrm{Prim}(A))$
continuous (when dropped canonically to Glimm$(A)$). Then Glimm$(A)$ becomes a
completely regular (Hausdorff) space called the \textit{complete regularization}
of Prim($A$), and the quotient map $\phi_A : \mathrm{Prim}(A) \to
\mathrm{Glimm}(A)$ is known as \textit{complete regularization map}.
\medskip

A \textit{derivation} on a $C^*$-algebra $A$ is a linear map $\delta : A \to A$
satisfying the \textit{Leibniz rule}
\begin{equation}\label{Leibniz}
\delta(ab)=\delta(a)b + a \delta(b), \quad \mathrm{for} \ \mathrm{all} \ a,b \in
A.
\end{equation}
The \textit{inner derivation} implemented by the element $a \in A$ is a map
$\delta_a : A \to A$, given by $$\delta_a (x):=ax-xa, \quad \mathrm{for }\
\mathrm{all} \  x \in A.$$
If a derivation $\delta \in \mathrm{Der}(A)$ is not inner, we say that $\delta$ is \textit{outer}.
By Der($A$) and Inn($A$) we respectfully denote the set of all derivations on $A$ and the set of all inner derivations on $A$.
It is well known that
$\mathrm{Der}(A) \subseteq \mathrm{ICB}(A)$ and
that
$$
\|\delta\|_{cb}=\|\delta\|=\sup \{\|\delta_P\| : \ P \in \mathrm{Prim}(A) \},
$$
where $\delta_J$ ($J \in \mathrm{Id}(A)$) denotes the induced derivation
on $A/J$; $$\delta_J(x+J)=\delta(x)+J \ \ (x \in A).$$
When $A$ is a primitive and unital $C^*$-algebra, $a \in A$ and $\lambda(a)$ the nearest scalar to $a$ (i.e. $\|a- \lambda(a)\|=d(a,\C)$), by Stampfli's formula \cite[4.1.17]{Ara} we have
\begin{equation}\label{Stampfli}
\|\delta_a\|_{cb}=\|\delta_a\|=2\|a-\lambda(a)\|.
\end{equation}

\section{Quasicentral $C^*$-algebras}

\begin{definition}
A $C^*$-algebra $A$ is said to be \textit{quasicentral} if no primitive ideal of
$A$ contains $Z(A)$. This is equivalent to the condition that no Glimm ideal of $A$ contains $Z(A)$.
\end{definition}
 The next proposition gives a useful characterization of quasicentral $C^*$-algebras:
\begin{proposition}\label{qc} Let $A$ be a $C^*$-algebra. The following
conditions are equivalent:
\begin{itemize}
\item[(i)] $A$ is quasicentral,
\item[(ii)] $A$ has a central approximate unit\footnote{We say that an approximate unit $(e_\alpha)$ of $A$ is central if $e_\alpha \in Z(A)$ for each $\alpha$.},
\item[(iii)] $A=Z(A)A$,
\item[(iv)] $A$ is Glimm ideal of $\tilde{A}$.
\end{itemize}
\end{proposition}
\begin{proof}
$(i) \Leftrightarrow (ii)$. This follows from \cite[Thm. 1]{Arc1}.

$(ii)\Rightarrow (iii).$ This follows directly from Cohen's factorization
theorem \cite[A.6.2]{Ble}, since $A$ is a nondegenerate Banach
$Z(A)$-modul.

$(iii)\Rightarrow (ii).$ This is trivial.

$(iii)\Leftrightarrow (iv).$ Since $\tilde{A}$ is unital, by \cite{ArcSom1} $A \in \mathrm{Glimm}(\tilde{A})$ if and only if there exists a maximal ideal $J$ of $Z(\tilde{A})$
such that $A=J\tilde{A}$. It follows that $J=Z(A)$, and hence $A\in \mathrm{Glimm}(\tilde{A})$ if and only if $A=Z(A)\tilde{A}=Z(A)A$.
\end{proof}

\begin{lemma}\label{triv} Let $A$ be a quasicentral $C^*$-algebra. Then $\mathrm{Inn}(A) \subseteq \im \theta_A$.
\end{lemma}
\begin{proof} By proposition \ref{qc}, each $a \in A$ can be written in the form $a=zb$, for some $z \in Z(A)$ and $b \in A$. It follows that $\delta_a=\theta_A(z\otimes b-b \otimes z)$.
\end{proof}
\begin{question} If $A$ is a $C^*$-algebra with the property that $\mathrm{Inn}(A) \subseteq \im \theta_A$, is $A$ necessarily quasicentral?
\end{question}

Let $A$ be a $C^*$-algebra. By Dauns-Hofmann theorem \cite[A.34]{RW}, there
exists an isomorphism $\Psi_A : Z(M(A)) \to C_b(\mathrm{Prim}(A))$ such that $$za
+ P= \Psi_A(z)(P)(a+P), \quad \mathrm{for} \ \mathrm{all} \ z \in Z(M(A)), \ a \in
A \ \mathrm{and} \ P \in \mathrm{Prim}(A).$$
Since the norm functions $P \mapsto \|a+P\|$ $(a
\in A)$, $\mathrm{Prim}(A) \to \R_+$ vanish at infinity (see \cite[4.4.4]{Ped}), we have $\Psi_A(Z(A))\subseteq
C_0(\mathrm{Prim}(A)).$
If  $A$ is quasicentral then it follows from \cite{Arc1} that
\begin{equation}\label{Z}
\Psi_A(Z(A))=C_0(\mathrm{Prim}(A))
\end{equation}
 Using (\ref{Z}) it is easy to prove the following fact:
\begin{proposition}\label{unit} Let $A$ be a quasicentral $C^*$-algebra. The
following conditions are equivalent:
\begin{itemize}\item[(ii)] $A$ is unital,
\item[(ii)] $\mathrm{Prim}(A)$ is compact.
\end{itemize}
\end{proposition}
\begin{proof} Implication $(i) \Rightarrow (ii)$ follows from \cite[3.1.8]{Dix}.

$(ii) \Rightarrow (i)$. If $\mathrm{Prim}(A)$ is compact, then by (\ref{Z}) we
have $Z(A) \cong C_0(\mathrm{Prim}(A))=C(\mathrm{Prim}(A))$.
Hence, $Z(A)$ is unital. By proposition \ref{qc} (iii) it follows that the unit
of $Z(A)$ must also be the unit of $A$.
\end{proof}

\begin{remark}\label{un}
If $A$ is a quasicentral $C^*$-algebra, it follows that for each $P \in
\mathrm{Prim}(A)$ there exists a positive element $z_P \in Z(A)_+$
such that $\|z_P\|=1$ and $\Psi_A(z_P)(P)=1$. It follows that each primitive quotient $A/P$ is
unital with the unit $z_P+P$. Moreover, using the Gelfand transform of $Z(A)$, it can be easily seen (like in the proof of \cite[Thm. 5]{Arc1}) that for each compact subset $K \subseteq \mathrm{Prim}(A)$ there
exists $z \in Z(A)_+$ such that $\|z\|=1$ and $\Psi_A(z)(P)=1$, for each $P \in
K$.
\end{remark}

\begin{lemma}\label{qct} Let $A$ be a quasicentral $C^*$-algebra and let $P,Q
\in \mathrm{Prim}(A)$. The following conditions are equivalent:
\begin{itemize}
\item[(i)] $P \approx Q$ (in a sense of (\ref{approx})),
\item[(ii)] $f(P)=f(Q)$, for all $f \in C_0(\mathrm{Prim}(A))$,
\item[(iii)] $P \cap Z(A)=Q \cap Z(A)$.
\end{itemize}
\end{lemma}
\begin{proof}
 Implications $(i)\Rightarrow (ii)$ and $(i)\Rightarrow (iii)$ follow
immediately.

$(ii)\Rightarrow (i)$. Let $g \in C_b(\mathrm{Prim}(A))$ and let
$f:=\Psi_A(z_P)$, where $z_P \in Z(A)_+$ is as in Remark \ref{un}. Then
 $f \in C_0(\mathrm{Prim}(A))$ and $f(P)=1$. By the assumption, we have $f(Q)=1$
and $(fg)(P)=(fg)(Q)$ (since $fg \in C_0(\mathrm{Prim}(A))$). Hence
$$g(P)=f(P)g(P)=(fg)(P)=(fg)(Q)=f(Q)g(Q)=g(Q).$$

$(iii)\Rightarrow (ii)$. Let $f \in C_0(\mathrm{Prim}(A))$. By (\ref{Z}) there
exists $z\in Z(A)$ such that $\Psi_A(z)=f$. Let $z_P, z_Q \in Z(A)_+$
 be as in Remark \ref{un}, and let $u:=\max\{z_P,z_Q\}$. Then for $v:=z-f(P)u$
we have $v \in P \cap Z(A)=Q \cap Z(A)$, and so
 $$0=\Psi_A(v)(Q)=f(Q)-f(P).$$
\end{proof}

If $A$ is unital, it follows from \cite{ArcSom1} that $\mathrm{Glimm}(A)$ is a compact Hausdorff space, and the map $\zeta_A: G \mapsto G \cap Z(A)$, from $\mathrm{Glimm}(A)$ onto $\mathrm{Max}(Z(A))$ is a homeomorphism, where $\mathrm{Max}(Z(A))$ denotes the maximal ideal space of $Z(A)$.
The next proposition gives a generalization of this result for quasicentral $C^*$-algebras.

\begin{proposition}\label{zz} Let $A$ be a quasicentral $C^*$-algebra. Then
$\mathrm{Glimm}(A)$ is a locally compact Hausdorff space and the map
$\zeta_A : G \mapsto G \cap Z(A)$ from $\mathrm{Glimm}(A)$ onto
$\mathrm{Max}(Z(A))$ is a homeomorphism.
\end{proposition}
\begin{proof} Let $G \in \mathrm{Glimm}(A)$. Since $A$ is quasicentral, there
exists $z \in Z(A)_+$ such that $\|z+G\|>0$. Since the function
$P \mapsto \|z+P\|=\Psi_A(z)(P)$ is continuous on $\mathrm{Prim}(A)$, it follows that
$$\{H \in \mathrm{Glimm}(A) : \ \|z+H\| \geq \|z+G\| \}$$
is a compact neighborhood of $G$ in $\mathrm{Glimm}(A)$, being a continuous
image under $\phi_A$ of a compact subset
$$\{P \in \mathrm{Prim}(A) : \ \|z+P\| \geq \|z+G\|\}$$
of $\mathrm{Prim}(A)$. Hence $\mathrm{Glimm}(A)$ is a locally compact Hausdorff
space. We now prove that $\zeta_A$ is a homeomorphism. Since each irreducible representation of $Z(A)$ can be
lifted to the irreducible representation of $A$ (see \cite[II.6.4.11]{Bla}), $\zeta_A$ is
surjective. That $\zeta_A$ is also injective follows from Lemma \ref{qct}
(iii). Since the topology of the locally compact Hausdorff space
$\mathrm{Glimm}(A)$ coincides with the weak topology induced by
$C_0(\mathrm{Glimm}(A))_+$ and since $C_0(\mathrm{Glimm}(A))_+=\Psi_A(Z(A)_+)$,
 for a net $(G_\alpha)$ in $\mathrm{Glimm}(A)$, and $G \in \mathrm{Glimm}(A)$ we
have

\begin{eqnarray*}
 G_\alpha \to G &\iff& \Psi_A(z)(G_\alpha) \to \Psi_A(z)(G), \
\mathrm{for} \ \mathrm{all} \ z \in Z(A)_+\\
&\iff& \|z+G_\alpha\|
\to \|z+G\|, \ \mathrm{for} \ \mathrm{all} \ z \in Z(A)_+ \\
&\iff&
\|z+ G_\alpha \cap Z(A)\| \to \|z+G \cap Z(A)\|, \ \mathrm{for} \ \mathrm{all} \
z \in Z(A)_+ \\
&\iff& G_\alpha \cap Z(A) \to G\cap Z(A).
\end{eqnarray*}
It follows that $\zeta_A$ is a homeomorphism.
\end{proof}

\begin{remark}
If $A$ is a non-unital quasicentral $C^*$-algebra, then by Proposition \ref{unit}
$\mathrm{Prim}(A)$ and (hence) $\mathrm{Glimm}(A)$
are non-compact spaces. For $J \in \mathrm{Id}(A)$ let $J_\sim$ be the unique ideal of $\tilde{A}$ such that $A \cap J_\sim=J$.
By Proposition \ref{qc} (iv) and Proposition \ref{zz} it follows that the map $G \mapsto G_\sim$ is a homeomorphism from
$\mathrm{Glimm}(A)$ onto its image
$\mathrm{Glimm}(\tilde{A})\setminus \{A\}$ in $\mathrm{Glimm}(\tilde{A})$. Since
$\tilde{A}$ is unital, $\mathrm{Glimm}(\tilde{A})$ is a compact Hausdorff space,
and hence $\mathrm{Glimm}(\tilde{A})$ is the Alexandroff compactification of
$\mathrm{Glimm}(A)$. Since $\zeta_{\tilde{A}}(A)=Z(A)$, we have the following commutative diagram:
$$\begin{CD}
\mathrm{Prim}(A) @> \phi_A >> \mathrm{Glimm}(A) @> \zeta_A >>
\mathrm{Max}(Z(A))\\
@V  VV@V  VV @V  VV\\
\mathrm{Prim}(\tilde{A}) @>\phi_{\tilde{A}}>> \mathrm{Glimm}(\tilde{A})
@>\zeta_{\tilde{A}}>> \mathrm{Max}(Z(\tilde{A})),
\end{CD}$$
where the vertical maps denote the canonical embeddings.
\end{remark}

\begin{proposition}\label{A+} Let $A$ be a $C^*$-algebra. The following
conditions are equivalent:
\begin{itemize}\item[(i)] $A$ is quasicentral and $\mathrm{Prim}(A)$ is
Hausdorff,
\item[(ii)] $\mathrm{Prim}(\tilde{A})$ is Hausdorff.
\end{itemize}
\end{proposition}
\begin{proof} $(i)\Rightarrow (ii)$. Since $\tilde{A}$ is unital, by Lemma \ref{qct} (iii)
it is sufficient to prove that distinct primitive ideals of
$\tilde{A}$ have distinct intersection with $Z(\tilde{A})$. Let $P,Q \in
\mathrm{Prim}(\tilde{A})$ such that $P \neq Q$. Then $P\cap A$ and $Q \cap A$
are distinct elements of $\mathrm{Prim}(A) \cup \{A\}$. Since $A$ is quasicentral and $\mathrm{Prim}(A)$
is Hausdorff, it follows from Lemma \ref{qct} (iii) that they have
distinct intersection with $Z(A) \subseteq Z(\tilde{A})$.

$(ii)\Rightarrow (i)$. Since $\mathrm{Prim}(\tilde{A})$ is Hausdorff, we have $\mathrm{Glimm}(\tilde{A})=\mathrm{Prim}(\tilde{A})$ and
hence $A \in \mathrm{Glimm}(\tilde{A})$. By Proposition \ref{qc} $A$ is
quasicentral. Since $\mathrm{Prim}(A)$ is homoeomorphic to the (open) subset $\mathrm{Prim}(\tilde{A})\setminus\{A\}$ of $\mathrm{Prim}(\tilde{A})$,  $\mathrm{Prim}(A)$ is also Hausdorff.
\end{proof}

\section{Derivations in $\im \theta_A$ on Prime $C^*$-algebras}

Recall that a $C^*$-algebra $A$ is called \textit{prime} if the zero ideal $(0)$
is a prime ideal of $A$. Since by \cite[1.2.47]{Ara} the center $Z(A)$ of a prime
$C^*$-algebra $A$ is either zero (if $A$ is non-unital) or isomorphic to $\C$
(if $A$ is unital), it follows from Proposition \ref{unit} that $A$ is unital if and only if
it is quasicentral.

\begin{remark}\label{rs}
Mathieu showed that the canonical contraction $\theta_A$ is an isometry if and only
if $A$ is prime $C^*$-algebra (see \cite[5.4.11]{Ara}). Since by \cite[1.1.7]{Ara} $A$ is prime if and only if $M(A)$ is prime, it follows
(using the Kaplansky's density theorem)
that in this case the map $$\Theta_A : M(A) \otimes_h M(A) \to \mathrm{ICB}(A), \quad \Theta_A(t):=\theta_{M(A)}(t)|_A$$
is also an isometry.
\end{remark}

Recall from \cite[3.2]{Smi} that a subset $\{a_n\}$ of a
$C^*$-algebra $A$ such that the series $\sum_{n=1}^\infty a_n^*a_n$
is norm convergent is said to be \textit{strongly independent} if
whenever  $(\alpha_n) \in \ell^2$ is a square summable sequence of complex
numbers such that $\sum_{n=1}^\infty \alpha_n a_n=0$, we have $\alpha_n=0$, for
all $n \in \N$.

The next lemma is a combination of \cite[1.5.6]{Ble}, \cite[4.1]{Smi} and
\cite[2.3]{ASS}.
\begin{lemma}\label{rep} Let $A$ be a $C^*$-algebra.
\begin{itemize}
\item[(a)] Every tensor $t \in A \otimes_h A$ has a representation as a
convergent series $t=\sum_{n=1}^\infty a_n \otimes b_n,$ where $(a_n)$ and
$(b_n)$ are sequences of $A$ such that the series $\sum_{n=1}^\infty a_n a_n^*$
and $\sum_{n=1}^\infty b_n^*b_n$ are norm convergent. Moreover, $\{b_n\}$ can be chosen to be strongly independent.
\item[(b)] If $t=\sum_{n=1}^\infty a_n \otimes b_n$ is a representation of $t$
as above, with $\{b_n\}$ strongly independent, then $t=0$ if and only if $a_n=0$,
for all $n \in \N$.
\end{itemize}
 \end{lemma}

\begin{theorem}\label{prime} Let $A$ be a prime $C^*$-algebra. Every derivation $\delta \in
\mathrm{Der}(A) \cap \im \theta_A$ is inner in $A$. If $A$ is
not unital, then $\mathrm{Der}(A) \cap \im \theta_A=\{0\}$.
 \end{theorem}
 \begin{proof}
Let $\Theta_A$ be the map as in Remark \ref{rs} and let $t \in A \otimes_h A$ be a tensor such that $\Theta_{A}(t)= \delta$ (we assume that $A\otimes_h A \subseteq M(A)\otimes_h M(A)$, by the injectivity of the Haagerup tensor product). Suppose that $t=\sum_{n=1}^\infty a_n \otimes b_n$ is a representation of
$t$ like in Lemma \ref{rep}, with  $\{b_n\}$ strongly independent. Since $\delta$
is a derivation on $A$, Leibniz rule (\ref{Leibniz}) implies that
\begin{equation}\label{D1}
\delta(x)y=\sum_{n=1}^\infty(a_nx-xa_n)yb_n, \quad  \mathrm{for} \ \mathrm{all}
\ x,y \in A.
\end{equation}
By remark \ref{rs} $\Theta_{A}$ is an isometry, and so the equality (\ref{D1}) is equivalent
to the equality of tensors
\begin{equation}\label{D2}
\delta(x) \otimes 1 = \sum_{n=1}^\infty (a_n x - xa_n) \otimes b_n  \ \
\mathrm{in} \ M(A) \otimes_h M(A), \quad  \mathrm{for} \ \mathrm{all}
\ x \in A.
\end{equation}
Suppose that $\delta \neq 0$. Then (\ref{D2}) implies that $A$ must be unital,
so $A=M(A)$. Indeed, choose $x_0 \in A$ such that $\delta(x_0)\neq
0$, and let $\varphi \in M(A)^*$ be an arbitrary (bounded) functional such that
$\varphi(\delta(x_0))\neq 0$. If we act in the equality (\ref{D2}) for $x=x_0$ with
the right slice map $R_\varphi$\footnote{For a $C^*$-algebra $B$ and $\psi \in B^*$, the \textit{right slice map} $R_\psi$ is a unique bounded map $ B\otimes_h B \to B$ given on elementary tensors by $R_\psi(a\otimes b)=\psi(a)b$ (see \cite[Section 4]{Smi}).}, we obtain
\begin{equation}\label{pom}
1=\frac{1}{\varphi(\delta(x_0))} \sum_{n=1}^\infty \varphi(a_n x_0 - x_0a_n)
b_n,
\end{equation}
and hence $1 \in A$. Let
$$\alpha_n:=\frac{\varphi(a_nx_0-x_0a_n)}{\varphi(\delta(x_0))} \ (n \in \N).$$
Since each bounded functional on a $C^*$-algebra is completely bounded (see \cite[3.8]{Paul}), and since the series
$\sum_{n=1}^\infty (a_nx_0-x_0a_n)(a_nx_0-x_0a_n)^*$ is norm convergent, it follows that $(\alpha_n)\in \ell^2$, and (\ref{pom}) implies that
$\sum_{n=1}^\infty \alpha_n b_n=1$. Then from (\ref{D2}) it follows that
$$\sum_{n=1}^\infty (\alpha_n \delta(x)-a_n x + xa_n)\otimes b_n=0, \quad
\mathrm{for} \ \mathrm{all} \ x \in A,$$
and consequently, since $\{b_n\}$ is strongly independent, Lemma \ref{rep} (b)
implies that
\begin{equation}\label{derr}
\alpha_n\delta(x)= a_n x - xa_n \quad \mathrm{for} \ \mathrm{all} \ x \in A \
\mathrm{and} \ n \in \N.
\end{equation}
Since $\sum_{n=1}^\infty \alpha_n b_n=1$, there is some $k \in \N$ such that
$\alpha_k \neq 0$. If $a:=\frac{a_k}{\alpha_k}$, then the equality (\ref{derr})
implies that $\delta = \delta_a \in \mathrm{Inn}(A)$.
\end{proof}

\section{Derivations in $\im \theta_A$ on $C^*$-algebras with Hausdorff primitive spectrum}

\begin{definition}\label{smooth} Let $A$ be a $C^*$-algebra, and let $\delta$ be a derivation on $A$. We define a
bounded function
$$|\delta| : \mathrm{Prim}(A) \to \R_+ \quad \mathrm{by} \quad
|\delta|(P):=\|\delta_{P}\|, \ \mathrm{for} \ P \in \mathrm{Prim}(A).$$
By \cite[2.2]{AEP} $|\delta|$ is a lower semi-continuous function on $\mathrm{Prim}(A)$.
If $|\delta|$ is continuous on $\mathrm{Prim}(A)$, we say that $\delta$ is \textit{smooth}.
\end{definition}

\begin{remark} The function $|\delta|$ is usually defined on the spectrum
$\hat{A}$ of $A$, by $|\delta|([\pi]):=\|\delta_{\pi}\|$, where
$\pi$ is some element of $[\pi]\in \hat{A}$, and $\delta_{\pi}$ denotes the
induced derivation on $\pi(A)$ ($\delta_\pi(\pi(a))=\pi(\delta(a))$ $(a \in A)$). In this case $\delta$ is said to
be smooth if $|\delta|$, as a function on $\hat{A}$, is continuous (see
\cite[2.3]{AEP} or \cite[4.2.6]{Ara}). Since $\|\delta_{\pi}\|=\|\delta_P\|$, where $P:=\ker\pi$, we note that this two definitions are
consistent with each other.
\end{remark}
The notion of the smooth derivation is important, since by \cite[2.4]{AEP} (or \cite[4.2.7]{Ara}) each smooth derivation on a $C^*$-algebra $A$ is inner in $M(A)$.
\medskip

 Let $A$ be a $C^*$-algebra and let $I,J \in \mathrm{Id}(A)$. If $q_I : A \to
A/I$ and $q_J : A \to A/J$ denote the quotient maps, it follows
 form \cite[2.8]{ASS} that the induced map
   $q_I \otimes q_J : A \otimes_h A \to A/I \otimes_h A/J$ is a quotient map and
that $$\ker(q_I \otimes q_J)=I \otimes_h A + A \otimes_h J.$$
Hence, we have
$$(A \otimes_h A)/(I \otimes_h A + A \otimes_h J) \cong A/I \otimes_h A/J,$$
isometrically.

We also define a
bounded function
 $$|t| : \mathrm{Prim}(A) \to \R_+ \quad \mathrm{by} \quad |t|(P):=\|q_P \otimes
q_P(t)\|_h, \ \mathrm{for} \ P \in \mathrm{Prim}(A).$$

Recall from \cite{Arc2} that the \textit{strong topology} $\tau_s$ on $\mathrm{Id}(A)$ is the weakest topology that makes all norm functions $J \mapsto \|a+J\|$ $(a \in A)$ continuous on $\mathrm{Id}(A)$.

\begin{lemma}\label{tens} Let $A$ be a $C^*$-algebra with  Hausdorff
primitive spectrum. For each tensor $t \in A \otimes_h A$ the function $|t|$ is continuous on $\mathrm{Prim}(A)$.
\end{lemma}
\begin{proof} Since $\mathrm{Prim}(A)$ is Hausdorff, by \cite[4.4.5]{Ped} the functions $P \mapsto
\|a+P\|$ $(a \in A)$ are continuous on $\mathrm{Prim}(A)$.
Hence, the Jacobson topology and the $\tau_s$-topology restricted to
$\mathrm{Prim}(A)$ coincide. By \cite[Prop. 2]{Som2} for each $t \in A \otimes_h A$
the map
$$\mathrm{Id}(A) \times \mathrm{Id}(A) \to \R_+, \quad (I,J) \mapsto \|t + (I
\otimes_h A + A \otimes_h J)\|=\|q_I \otimes q_J(t)\|_h$$
is continuous for the product $\tau_s$-topology on $\mathrm{Id}(A) \times
\mathrm{Id}(A)$.
If $D$ denotes the diagonal of $\mathrm{Prim}(A)\times \mathrm{Prim}(A)$, the map $(P ,P)\mapsto\|q_P \otimes q_P(t)\|_h=|t|(P)$ is
continuous on $D$, and so the map $|t|$ is continuous on  $\mathrm{Prim}(A)$.
\end{proof}

\begin{remark}\label{comm} Let $A$ be a $C^*$-algebra.
It is easy to check that for all $J \in \mathrm{Id}(A)$ the following diagram
$$
\begin{CD}
A \otimes_h A @> \theta_A>> \mathrm{ICB}(A)\\
@V q_J \otimes q_J VV@V Q_J VV\\
A/J \otimes_h A/J @>\theta_{A/J}>> \mathrm{ICB}(A/J),
\end{CD}
$$
commutes, where $Q_J$ denotes the induced map $Q_J : \mathrm{ICB}(A) \to
\mathrm{ICB}(A/J)$,
\begin{equation}\label{Q_J}
  Q_J(T)(q_J(x)):=q_J(T(x)), \ \mathrm{for} \ \mathrm{all} \ T\in
\mathrm{ICB}(A) \ \mathrm{and} \ x\in A.
\end{equation}
Hence, if $\delta \in \mathrm{Der}(A) \cap \im \theta_A$ and $t \in A \otimes_h A$
such that $\delta=\theta_A(t)$, we have
\begin{equation}\label{D_P}
\delta_{J}=Q_J(\theta_A(t))=\theta_{A/J}(q_J \otimes q_J(t)).
\end{equation}
\end{remark}
\begin{remark}\label{pros} Let $A$ be a $C^*$-algebra and let $\delta \in \mathrm{Der}(A) \cap \im \theta_A$, with $\delta=\theta_A(t)$, for
 some tensor $t \in A \otimes_h A$. If we embed $A$ into the von Neumann
envelope $A^{**}$ of $A$, then by \cite[4.2.3]{Ara} $\delta$ can be extended (by ultraweak continuity) to the derivation $\delta^{**}$ on $A^{**}$.
It follows that $\delta^{**}=\theta_{A^{**}}(t)$ (where $A \otimes_h A \subseteq A^{**}\otimes_h A^{**}$, by the injectivity of the Haagerup tensor product), and hence $\tilde{\delta}=\delta^{**}|_{\tilde{A}}=\theta_{\tilde{A}}(t)$, where $\tilde{\delta}$ denotes the (unique) extension of $\delta$ to the derivation on the minimal
unitization $\tilde{A}$ of $A$.
\end{remark}
\begin{theorem}\label{Hauss} Let $A$ be a $C^*$-algebra with Hausdorff primitive
spectrum. Every derivation $\delta \in \im\theta_A$
is smooth and hence inner in $M(A)$. Moreover, if $A$ is also quasicentral, then $\delta$ is inner
in $A$.
\end{theorem}
\begin{proof}
First, we will show that $\delta$  is smooth. Let
$t \in A \otimes_h A$ be a tensor such that $\delta= \theta(t)$, and let $P \in
\mathrm{Prim}(A)$. By (\ref{D_P}) we have $\delta_{P}=\theta_{A/P}(q_P \otimes
q_P(t))$. Since $A/P$ is primitive (simple in fact, since $\mathrm{Prim}(A)$ is
Hausdorff), $\theta_{A/P}$ is an isometry, and hence
$$|\delta|(P)=\|\delta_P\|=\|\delta_P\|_{cb}=\|\theta_{A/P}(q_P \otimes
q_P(t))\|_{cb}=\|q_P \otimes q_P(t)\|_h=|t|(P).$$
Since $P \in \mathrm{Prim}(A)$ was arbitrary, Lemma \ref{tens} implies
that $|\delta|=|t|$ is a continuous function on $\mathrm{Prim}(A)$, and hence,
$\delta$ is smooth. By \cite[2.4]{AEP} (or \cite[4.2.7]{Ara}) there exists an element $b \in M(A)$
such that $\delta=\delta_b$.

Now suppose that $A$ is also quasicentral, and let $\tilde{\delta}$ be
the (unique) extension of $\delta$ to the derivation on $\tilde{A}$. By Remark \ref{pros} we have $\theta_{\tilde{A}}(t)=\tilde{\delta}$.
Since $\mathrm{Prim}(\tilde{A})$ is also Hausdorff (Proposition \ref{A+}), by the first part of the
proof there exists $b \in \tilde{A}$ which induces $\tilde{\delta}$. If we
choose $\alpha \in \C$ such that $a:=b-\alpha 1 \in A$, then obviously $a$ also
induces $\tilde{\delta}$, and hence $\delta=\tilde{\delta}|_A$ is inner in $A$.
\end{proof}

\begin{question} Can one always (without the assumption of quasicentrality)
conclude that $\mathrm{Der}(A) \cap \im \theta_A \subseteq \mathrm{Inn}(A)$, when
$\mathrm{Prim}(A)$ is Hausdorff?
\end{question}

\begin{corollary}\label{kar1} Let $A$ be a $C^*$-algebra.
\begin{itemize}
\item[(i)] If $A$ is quasicentral and $\mathrm{Prim}(A)$ is Hausdorff then each inner derivation on $A$ is smooth.
\item[(ii)] If each inner derivation on $A$ is smooth then $\mathrm{Prim}(A)$ is Hasudorff.
\end{itemize}
Hence, for a quasicentral $C^*$-algebra $A$, $\mathrm{Prim}(A)$ is Hausdorff if and only if each inner derivation on $A$ is smooth.
\end{corollary}
\begin{proof} $(i)$. Since $A$ is quasicentral, by Lemma \ref{triv} $\mathrm{Inn}(A) \subseteq \im \theta_A$, so by Theorem \ref{Hauss} each inner derivation on $A$ is smooth.

$(ii)$. Let $a \in A_h$. Since $\delta_a$ is smooth, by \cite[2.10]{AEP} the function $P \mapsto \|(a+z)+P^{\sim}\|$ is continuous on $\mathrm{Prim}(A)$, for each $z \in Z(M(A))_h$, where $P^{\sim}$ $(P \in \mathrm{Prim}(A))$ denotes the unique primitive ideal of $M(A)$ such that $A \cap P^{\sim}=P$.  Hence, for
$z=0$, the function $P \mapsto \|a+P^\sim\|=\|a+P\|$ is continuous on $\mathrm{Prim}(A)$, and since $a\in A_h$ was arbitrary, by \cite[4.4.5]{Ped} $\mathrm{Prim}(A)$ is Hausdorff.
\end{proof}

The result of Corollary \ref{kar1} is not true in general for non-quasicentral
$C^*$-algebras, even if $\mathrm{Prim}(A)$ is Hausdorff and
every primitive quotient of $A$ is unital.
\begin{example} Let $A$ be a $C^*$-algebra consisting of all continuous functions
$a : [0,1] \to \MM_2(\C)$ such that $$a(1)=\left( \begin{array}{cc}
               \lambda(a) & 0 \\
              0 & 0 \\
               \end{array} \right), \quad \mathrm{for} \ \mathrm{some} \
\lambda(a) \in \C. $$
It is easy to check that every irreducible representation of $A$ is equivalent
to some representation $\pi_t$ $(t \in [0,1])$, where $\pi_t : a \mapsto a(t)$,
for $t \in  [0,1)$, and $\pi_1 : a \mapsto\lambda(a)$, and that
the map $t \mapsto P_t:=\ker\pi_t$ is a homeomorphism from $[0,1]$ onto
$\mathrm{Prim}(A)$. Since
$$Z(A)=\Big \{\left( \begin{array}{cc}
               f & 0 \\
              0 & f \\
               \end{array} \right) : \ f \in C_0([0,1)) \Big\} \subseteq P_1,$$
$A$ is not quasicentral. Let $a$ be an element of $A$ such that
$$a(t)=\left( \begin{array}{cc}
               1 & 0 \\
              0 & 0 \\
               \end{array} \right), \ \ \mathrm{for} \ \mathrm{all} \ t \in [0,1],$$
and let $\delta:=\delta_a$.
By Stampfli's formula (\ref{Stampfli}) we have
$$\|\delta_ {P_t}\|=2d(a+P_t,\C)=\left\{\begin{array}{cl}
1, & \textrm{ if } 0\leq t  <1 ,\\
0, & \textrm{if } t=1
\end{array}\right.$$
and hence, $\delta$ is not smooth.

\end{example}

\section{An example of a $C^*$-algebra with outer elementary derivations}

In this section we shall give an example of a unital $C^*$-algebra $A$ which has
an outer derivation $\delta \in \im \theta_A$. For this $C^*$-algebra $A$ the
space $\mathrm{Inn}(A)$ is not closed in the space $\mathrm{Der}(A)$.
By \cite[4.6]{Som3} this happens if and only if $\mathrm{Orc}(A)=\infty,$ where
$\mathrm{Orc}(A)$ is a constant arising from a certain graph structure on
$\mathrm{Prim}(A)$ which is defined as follows.

We say that two primitive ideals $P ,Q \in \mathrm{Prim}(A)$ are
\textit{adjacent} (and write $P\sim Q$) if $P$ and $Q$ cannot be separated by
disjoint open
subsets of $\mathrm{Prim}(A)$. A \textit{path} of length $n$ from $P$ to $Q$ is
a sequence of points $P=P_0,P_1, \ldots, P_n=Q$ such that $P_{i-1} \sim P_i$,
for all $1\leq i \leq n$. The \textit{distance} $d(P,Q)$ from $P$ to $Q$ is defined as follows:
\begin{itemize}
\item[-] If $P=Q$, $d(P,Q)=d(P,P):=1$,
\item[-] If $P\neq Q$ and there exists a path from $P$ to $Q$, then $d(P,Q)$ is equal to the minimal length of a path from
$P$ to $Q$.
\item[-] If there is no path from $P$ to $Q$, $d(P,Q):=\infty$.
\end{itemize}
The \textit{connecting
order} $\mathrm{Orc}(A)$ of $A$ is defined by
$$\mathrm{Orc}(A):= \sup\{d(P,Q) : \ P,Q \in \mathrm{Prim}(A) \ \mathrm{such} \
\mathrm{that} \ d(P,Q) < \infty \}.$$
Note that $\mathrm{Orc}(A)=1$ if $\mathrm{Prim}(A)$ is Hausdorff, but that the converse does not hold in general\footnote{As noted in \cite{Som2}, $\mathrm{Orc}(A)=1$ if and only if every Glimm ideal of $A$ is $2$-primal.}.

\medskip
We shall also use the following notation. Let $B$ be a unital $C^*$-algebra and
let $A \subseteq B$ be a $C^*$-subalgebra of $B$. An \textit{elementary operator} on $B$
\textit{with the coefficients} in $A$ is a map $T: B \to B$ which can be expressed in the
form
$$
T=\sum_{k=1}^d a_k \odot b_k, \quad  \mathrm{for} \ \mathrm{some } \ a_k,b_k \in
A \ (1\leq k \leq d),
$$
where
$$\Big(\sum_{k=1}^d a_k \odot b_k\Big)(x):=\theta_B\Big( \sum_{k=1}^d a_k \otimes b_k\Big)(x)=\sum_{k=1}^d a_k x b_k \quad (x \in
B).
$$
The space of all elementary operators on $B$ with the coefficients in $A$ is
denoted by $\E_A(B)$. If $A=B$ then (as usual) we write $\E(B)$ for $\E_B(B)$;
the set of all elementary operators on $B$. We also denote by $\E(B\rightarrow A)$ the
subspace of all $T \in \E(B)$ such that $T(B) \subseteq A$.

\begin{example}\label{orc} Let $\tilde{X}:=[1, \infty]$ be the Alexandroff
compactification of the interval $X:= [1, \infty)$, let
$B:=C(\tilde{X},\MM_2(\C))$, and let $A$ be a $C^*$-subalgebra of $B$ which
consists of all $a \in B$ such that
$$a(n)=\left( \begin{array}{cc}
               \lambda_n(a) & 0 \\
              0 & \lambda_{n+1}(a)\\
               \end{array} \right) \  (n \in \N) \ \ \mathrm{and} \ \ a(\infty)=\left( \begin{array}{cc}
               \lambda(a) & 0 \\
              0 & \lambda(a)\\
               \end{array} \right),$$
for some convergent sequence $(\lambda_n(a))$ of complex
numbers with $\lim_n \lambda_n(a)=\lambda(a)$.
Then $\mathrm{Orc}(A)=\infty$ and $\E(A)$ is a cb-closed subspace of
$\mathrm{ICB}(A)$. Consequently, $A$ has outer elementary derivations\footnote{A derivation $\delta\in \mathrm{Der}(A)$ is said to be elementary if
$\delta$ is an elementary operator on $A$.}.
\end{example}
This example is just a slightly modified version of the $C^*$-algebra
$A(\infty)$ in \cite[2.8]{Som3}. It is easy to check that
$$\mathrm{Prim}(A)=\{P_t :  \ t \in X \setminus \N\} \cup \{Q_n : \ n \in \N\}
\cup\{ Q\},$$
where $P_t$ $(t \in X \setminus \N)$ denotes a kernel of $a \mapsto a(t)$,
$Q_n$ $(n \in \N)$ denotes a kernel of $a \mapsto \lambda_n(a)$, and $Q$ denotes
the kernel of $a \mapsto \lambda(a)$. Also note that the points $P_t$ $(t \in X
\setminus \N)$ and $Q$ are separated\footnote{We say that a point $P\in \mathrm{Prim}(A)$ is separated if whenever $Q \in \mathrm{Prim}(A)$ and $P \nsubseteq Q$ then there exist disjoint open neighborhoods of $P$ and $Q$ in $\mathrm{Prim}(A)$.} in $\mathrm{Prim}(A)$, while $Q_i \sim Q_j$
if and only if $|i-j|\leq 1$. It follows that $d(Q_1,Q_{n+1})=n$, for all $n \in
\N$, and hence $\mathrm{Orc}(A)=\infty$. By \cite[4.6]{Som3}
$\mathrm{Inn}(A)$ is not closed in $\mathrm{Der}(A)$. One can also check this by
direct calculations. For example, it is not difficult to see that for each
function $f \in C_0(X)$ such that the series $\sum_{n=1}^\infty f(n)$ does not
converge, then the element $$b=\left( \begin{array}{cc}
               f & 0 \\
              0 & 0\\
               \end{array} \right) \in B$$
derives $A$ (that is $ab-ba \in A$, for all $a \in A$) and the induced
derivation (which is obviously not inner in $A$) is in the
closure of $\mathrm{Inn}(A)$. To prove that $\E(A)$ is closed in
$\mathrm{ICB}(A)$ we shall first need some additional technical results which will be stated in a more general setting.
\medskip

Let $A$ be a $C^*$-algebra. Recall that $A$ is called $n$-homogeneous $(n \in
\N)$ if $\dim \pi =n$, for all $[\pi] \in \hat{A}$. Then by \cite[3.2]{Fell}
$\Delta:=\mathrm{Prim}(A)$ is a (locally compact) Hausdorff space and $A$ is
isomorphic to the $C^*$-algebra $\Gamma_0(E)$ of all continuous sections
vanishing at infinity of a locally trivial $C^*$-bundle $E$ over $\Delta$ with
fibres isomorphic to $\MM_n(\C)$. If the base space $\Delta$ of $E$
admits a finite open covering
$\{U_j\}$ such that each $E|_{U_j}$ is trivial (as a $C^*$-bundle) we say that
$E$ (and hence $A$) is of \textit{finite type}.

If $$\sup\{\dim \pi : \ [\pi] \in \hat{A}\}=n$$
then we say that $A$ is $n$-\textit{subhomogeneous}. In this case $$J:=\bigcap
\{\ker \pi : \ [\pi] \in \hat{A} \ \mathrm{such} \ \mathrm{that} \  \dim \pi < n
\}$$
is called $n$-\textit{homogeneous ideal} of $A$, and is the largest ideal of $A$
which is $n$-homogeneous, as a $C^*$-algebra.
\begin{remark} If $A$ is $n$-subhomogeneous $C^*$-algebra, note that for each operator $T \in \im \theta_A$ we have $$\|T\|_{cb} \leq n \|T\|.$$
Indeed, this can be easily seen by using the formulas
 \begin{equation}\label{opformula}
\|T\|=\sup\{\|T_P\| : \ P \in \mathrm{Prim}(A)\}, \quad \mathrm{and} \quad
\|T\|_{cb}=\sup\{\|T_P\|_{cb} : \ P \in \mathrm{Prim}(A)\},
\end{equation}
(see \cite[5.3.12]{Ara}) and noting that each operator $S : \MM_m(\C) \to \MM_m(\C)$ is completely bounded (elementary in fact) with
$\|S\|_{cb} \leq m\|S\|$ (see \cite[Exercise\,3.11]{Paul}). Hence, if $A$ is subhomogeneous, we do not have to specify which norm do we
consider when speaking about closures of $\im \theta_A$ or $\E(A)$.
\end{remark}

\begin{lemma}\label{E(B;J)} Let $B$ be a unital $n$-homogeneous $C^*$-algebra,
and let $J \in \mathrm{Id}(B)$. Then
$\E_J(B)=\E(B\rightarrow J)$. In particular, $\E_J(B)$ is a closed subspace of $\E(B)$.

\end{lemma}

\begin{proof} Let $E$ be a locally trivial $C^*$-bundle $E$ over
$\Delta:=\mathrm{Prim}(A)$ (which is compact since $A$ is unital) whose fibres
are isomorphic to $\MM_n(\C)$ such that
$A=\Gamma(E)$ (we identify $A$ with $\Gamma(E)$ via the fixed isomorphism). By
compactness of $\Delta$ and local triviality of $E$, there exists a
finite open cover $\{U_j\}_{1\leq j \leq m}$ of $\Delta$
such that each $E|_{\overline{U_j}}$ is trivial. Using a finite partition of
unity, subordinated to the cover
$\{U_j\}_{1\leq j \leq m}$ one can reduce the proof to the situation when $m=1$,
so we may assume $E$ is trivial. Then $B=C(\Delta, \MM_n(\C))$, and since $J$ is
an ideal of $B$,
there is a closed subset $Y$ of $\Delta$ such that
$$J=\{a \in B : \ a|_Y=0\}.$$
Let $(E_{i,j})_{1\leq i,j \leq n}$ denote the standard matrix units of $\MM_2(\C)$ considered as constant elements of $B=C(\Delta,
\MM_n(\C))$, and let $T \in \E(B\rightarrow J)$. Then $T$ can be written in the form
\begin{equation}\label{koefic}
T=\sum_{i,j,p,q=1}^n f_{i,j,p,q} E_{i,j} \odot E_{p,q},
\end{equation}
for some functions $f_{i,j,p,q} \in C(\Delta)\cong Z(B)$. Let $1\leq r,s \leq n$
be the fixed numbers. Since $T(B) \subseteq J$, we have
$$T(E_{r,s})=\sum_{i,j,p,q=1}^n f_{i,j,p,q} E_{i,j} E_{r,s}
E_{p,q}=\sum_{i,q=1}^n f_{i,r,s,q}E_{i,q} \in J.$$
Thus, $f_{i,r,s,q}|_{ Y}=0$, for all $i,q=1, \ldots , n$. Since $r,s$ were
arbitrary, we have $$f_{i,j,p,q}|_ {Y}=0, \quad \mathrm{for} \ \mathrm{all} \
1\leq i,j,p,q\leq n$$
Note that every function $f \in C(\Delta)$ with the property $f|_{Y}=0$ can be
factorized in the form $f=gh$, where
 $g,h\in C(\Delta)$ such that $g|_{Y}=0$ and $h|_{Y}=0$ (for example, put
$g:=\sqrt{|f|}$ i $h:=f/\sqrt{|f|}$).
If we apply this factorization to the functions $f_{i,j,p,q}$, $$f_{i,j,p,q}=
g_{i,j,p,q}\cdot h_{i,j,p,q},$$ then it follows from (\ref{koefic})  that
$$T=\sum_{i,j,p,q=1}^n f_{i,j,p,q} E_{i,j} \odot E_{p,q}=\sum_{i,j,p,q=1}^n
g_{i,j,p,q}E_{i,j} \odot h_{i,j,p,q}E_{p,q}.$$
Thus $T \in \E_J(B)$.
\end{proof}

\begin{remark}\label{exact} Suppose that $$0 \longrightarrow X \longrightarrow Y \stackrel{q}\longrightarrow Z \longrightarrow 0$$
is an exact sequence of normed spaces, where $q$ is a bounded linear map. If $q$ is also open, note that $Y$ is a Banach space if and only if $X$ and $Z$ are Banach spaces. Also note that if
$\dot{Y} \subseteq Y$ and $\dot{Z} \subseteq Z$  are (not necessarily closed) subspaces such that $q(\dot{Y})=\dot{Z}$ and which fit into the exact sequence
$$0 \longrightarrow X \longrightarrow \dot{Y} \stackrel{\dot{q}}\longrightarrow \dot{Z} \longrightarrow 0,$$
where $\dot{q}:=q|_{\dot{Y}}$ (and hence $\dot{Y}=\dot{q}^{-1}(\dot{Z})=q^{-1}(\dot{Z})$), then $\dot{q}$ is open whenever $q$ is open.
\end{remark}
\begin{lemma}\label{str} Suppose that $A$ is a unital $n$-subhomogeneous
$C^*$-algebra with $n$-homogeneous ideal $J$ which is of finite type. If $B$ is any unital $n$-homogeneous $C^*$-algebra
which contains $A$ and such that $J$ is the essential ideal of $B$, then $\E(A)$ is closed subspace of $\mathrm{ICB}(A)$ if and
only if $\E_{A/J}(B/J)$ is a closed subspace of $\mathrm{ICB}(B/J)$.
\end{lemma}
\begin{proof} First note that $J$ is also essential in $A$. Also note that such $B$ exists, since by \cite{Mag} $M(J)$ is
$n$-homogeneous, and $A \subseteq M(J)$, since $J$ is essential in $A$. By
Kaplansky's density theorem the restriction map $T \mapsto T|_A$  is an
isometric isomorphism from $\mathrm{E}_A(B)$ onto $\E(A)$. Hence, we may
identify $\E(A)$ with $\E_A(B)$. Let $q_J : B \to B/J$ be a quotient map, and
let $\dot{Q}_J$ be the restriction of the induced contraction $Q_J$ to $\E(B)$
(see (\ref{Q_J})). Obviously $\dot{Q}_J(\E(B))=\E(B/J)$ and the kernel of
$\dot{Q}_J$ is the set $\E(B\rightarrow J)$, which can be identified with the set $\E_J(B)$,
by Lemma \ref{E(B;J)}. Since $B$ and $B/J$ are unital
homogeneous $C^*$-algebras, by \cite[1.1]{Mag} we have equalities
$\mathrm{ICB}(B)=\E(B)$ and $\mathrm{ICB}(B/J)=\E(B/J)$. Thus $\E(B)$ and
$\E(B/J)$ are Banach spaces, and by the open mapping theorem, $\dot{Q}_J$ is an
open map. Since $\dot{Q}_J(\E_A(B))=\E_{A/J}(B/J)$, note that the exact sequence
$$0 \longrightarrow \E_J(B) \longrightarrow \E(B)
\stackrel{\dot{Q}_J}\longrightarrow \E(B/J) \longrightarrow 0$$
of Banach spaces induces the exact sequence of normed spaces
 $$0 \longrightarrow \E_J(B) \longrightarrow \E_A(B)
\stackrel{\ddot{Q}_J}\longrightarrow \E_{A/J}(B/J) \longrightarrow 0,$$
where $\ddot{Q}_J$ denotes a restriction of $\dot{Q}_J$ to the set
$\E_A(B)$, since $\ker \ddot{Q}_J = \ker \dot{Q}_J= \E_J(B)$. By Remark \ref{exact}, $\ddot{Q}_J$
is also an open map, and since $\E_J(B)$ is a Banach space (Lemma \ref{E(B;J)}), $\E_A(B)$ is a Banach space if and only if $\E_{A/J}(B/J)$
is a Banach space.
\end{proof}
Now we prove the second claim of the example \ref{orc}.

\begin{lemma} Let $A$ and $B$ be the $C^*$-algebras from the Example \ref{orc}. Then
$\E(A)$ is a closed subspace of $\mathrm{ICB}(A)$.
\end{lemma}
\begin{proof}  Let
 $$J:=\{a \in A : \ a(n)=0, \ \mathrm{for} \ \mathrm{all} \  n \in \N\}$$ be the
$2$-homogeneous (Glimm) ideal of $A$.
 Then $J$ is an essential ideal of $A$ and $B$, and it follows from Lemma
\ref{str} that it is sufficient to show that $\E_{A/J}(B/J)$ is a closed
subspace of $\mathrm{ICB}(B/J)$ which is equal to $\E(B/J)$, by \cite[1.1]{Mag}.
Let
$$\dot{B}:= C(\tilde{\N,} \MM_2(\C)) \quad \mathrm{and} \quad \ \dot{A}:=\Big\{\left( \begin{array}{cc}
               f & 0 \\
              0 & \tilde{f}\\
               \end{array} \right): \ f \in C(\tilde{\N})\Big\},$$
where $\tilde{\N:}=\N \cup \{\infty\}$ denote the Alexandroff compactifcation of
$\N$, and for $f \in C(\tilde{\N})$, $\tilde{f}$ is a function defined with
$\tilde{f}(n):=f(n+1)$ $(n \in \N)$. Obviously $B/J \cong \dot{B}$ and $A/J
\cong \dot{A}$, and in the following, we shall identify this $C^*$-algebras. If
$(E_{i,j})_{1\leq i,j\leq 2}$ denote the standard matrix units of $\MM_2(\C)$ considered as constant elements of $\dot{B}$, we
claim that the set $\E_{\dot{A}}(\dot{B})$
can be identified with the set of all operators $T \in \E(\dot{B})$ which can be
written in the form
\begin{equation}\label{form}
T =f E_{1,1} \odot E_{1,1} + g E_{1,1} \odot E_{2,2}+ h E_{2,2} \odot E_{1,1} +
\tilde{f} E_{2,2} \odot E_{2,2}
\end{equation}
 where $f,g,h\in C(\tilde{\N})$ are functions such that
$$L(T):=f(\infty)=g(\infty)=h(\infty).$$ One can easily show that every $T \in
\E_{\dot{A}}(\dot{B})$ can be written in the form (\ref{form}). Conversely, if
$T\in \E(\dot{B})$ is of the form  (\ref{form}), then
\begin{eqnarray*}
T &=& (f-L(T)) E_{1,1} \odot E_{1,1} + (g -L(T))E_{1,1} \odot E_{2,2} \\
 &+&(h-L(T))
E_{2,2} \odot E_{1,1} + (\tilde{f}-L(T)) E_{2,2} \odot E_{2,2} +L(T)\mathrm{Id},
\end{eqnarray*}

where $\mathrm{Id}$ denotes the identity map on $\dot{B}$.
Hence, to prove that $T \in \E_{\dot{A}}(\dot{B})$, it is sufficient to prove that for
arbitrary functions $f,g,h \in C_0(\N)$ all operators $T_1,T_2$ and $T_3$ are
the elements of $\E_{\dot{A}}(\dot{B})$, where
$$T_1:=f E_{1,1} \odot E_{1,1} + \tilde{f}E_{2,2} \odot E_{2,2}, \quad T_2:=g
E_{1,1} \odot E_{2,2} \quad \mathrm{and} \quad  T_3:=h E_{2,2} \odot E_{1,1}.$$

\textit{Claim 1.} $T_1$ can be written in the form $$T_1 = a_1 \odot b_1 + a_2 \odot
b_2, \quad \mathrm{for} \ \mathrm{some}\  a_i,b_i \in \dot{A}.$$

To prove this, by looking at the entries of the corresponding decomposition of $T_1$, it is sufficient to find two sequences of vectors $(\vec{v}_n)$ and $(\vec{w}_n)$ in $\C^2$ such that
$\lim_n \vec{v}_n = \lim_n \vec{w}_n = (0,0)$, and
\begin{equation}\label{niz1}
\vec{v}_n \cdot \vec{w}^*_n = f(n), \quad \vec{v}_{n} \cdot
\vec{w}^*_{n+1}=\vec{v}_{n+1} \cdot \vec{w}^*_n=0, \quad \mathrm{for} \ \mathrm{all} \ n \in \N,
\end{equation}
where $\cdot$ denotes a standard inner product of $\C^2$, and for
$\vec{v}=(\alpha, \beta) \in \C^2$, $\vec{v}^*:=(\bar{\alpha}, \bar{\beta})$.
Let $\varphi,\psi \in C_0(\N)$ be any functions such that $f=\varphi \psi$. Then
we can achieve (\ref{niz1})
by putting $$\vec{v}_n=\Big([n+1] \varphi(n),
[n] \varphi(n)\Big ) \quad \mathrm{and} \quad  \vec{w}_n=\Big ([n+1]
\psi(n), [n]\psi(n) \Big) \ (n \in \N)$$
where $[n]=1$ if $n$ is even and $[n]=0$ if $n$ is odd.

\textit{Claim 2.} $T_2$ can be written in the form $$T_2 = a_1 \odot b_1 + a_2 \odot
b_2+ a_3 \odot b_3, \quad \mathrm{for} \ \mathrm{some}\  a_i,b_i \in
\dot{A}.$$

To prove this, like in the proof of Claim 1, it is sufficient to find two sequences of vectors
$(\vec{v}_n)$ and $(\vec{w}_n)$ in $\C^3$ such that
$\lim_n \vec{v}_n = \lim_n \vec{w}_n = (0,0,0)$, and
\begin{equation}\label{niz2}
\vec{v}_n \cdot \vec{w}^*_n = \vec{v}_{n+1} \cdot \vec{w}^*_n= 0, \quad
\vec{v}_{n} \cdot \vec{w}^*_{n+1}=g(n), \quad \mathrm{for} \ \mathrm{all} \ n \in \N.
\end{equation}
Let $\varphi,\psi \in C_0(\N)$ be any functions such that $g=\varphi \psi$. If $(\vec{e}_i)_{1\leq i \leq 3}$ denote the canonical basis of $\C^3$,
we can achieve (\ref{niz2}) by putting
$$\vec{v_n}= \varphi(n) \vec{e}_{\langle n \rangle}  \quad \mathrm{ and} \quad   \vec{w_i}= \psi(n-1) \vec{e}_{\langle n+2 \rangle} \ (n \in \N),$$
where $\psi(0):=1$, and for $n=3k+l$, $\langle n \rangle =l$ if $l=1,2$ and $\langle n \rangle=3$ if $l=0$.

\textit{Claim 3.} $T_3$ can be written in
the form $$T_3 = a_1 \odot b_1 + a_2 \odot b_2+ a_3 \odot b_3, \quad
\mathrm{for} \ \mathrm{some}\  a_i,b_i \in \dot{A}.$$

This can be proved like claim 2.

Using (\ref{form}) it is now easy to vertify that $\E_{\dot{A}}(\dot{B})$ is closed
in $\mathrm{ICB}(\dot{B})=\E(\dot{B})$.
\end{proof}
\begin{question} Does every unital $C^*$-algebra $A$ with
$\mathrm{Orc}(A)=\infty$ has an outer elementary derivation, or at least an
outer derivation from $\im \theta_A$?
\end{question}

Let $A$ be a separable $C^*$-algebra, and let $J \in \mathrm{Id}(A)$. By \cite[8.6.15]{Ped}
we know that each derivation $\dot{\delta} \in \mathrm{Der}(A/J)$ can be lifted
to the derivation $\delta \in \mathrm{Der}(A)$. Obviously, each operator
$\dot{T} \in \im \theta_{A/J}$ can also be lifted to an operator $T \in \im
\theta_A$. The next example shows that in general we cannot expect that a
derivation $\dot{\delta} \in \mathrm{Der}(A/J) \cap \im\theta_{A/J}$ has a lift
to a derivation $\delta \in \mathrm{Der}(A) \cap \im \theta_A$.

\begin{example} Let $A$ be the $C^*$-algebra from the Example \ref{orc} and
choose any faithful unital representation $ \pi : A \to  \BB(\mathcal{H})$ on a
separable Hilbert space $\mathcal{H}$ such that $\pi(A) \cap
\mathrm{K}(\mathcal{H})=\{0\}$, where $\mathrm{K}(\mathcal{H})$ denote the
$C^*$-algebra of all compact operators on $\mathcal{H}$. Let
$B:=\pi(A)+\mathrm{K}(\mathcal{H})$. Obviously $B$ is a unital, separable and primitive
$C^*$-algebra and hence, by Theorem \ref{prime}, we have $\mathrm{Der}(B) \cap \im
\theta_B = \mathrm{Inn}(B)$. On the other hand, since
$$B/\mathrm{K}(\mathcal{H}) \cong \pi(A)/( \pi(A)\cap \mathrm{K}(\mathcal{H}))
\cong \pi(A) \cong A,$$
by Example \ref{orc}  there exists an outer derivation $\dot{\delta} \in \im
\theta_{B/\mathrm{K}(\mathcal{H})}$. It follows that such derivation cannot be lifted
to the (necessarily inner) derivation $\delta \in \im \theta_B$.

\end{example}

\end{document}